\documentclass[12pt]{amsart}

\usepackage[american]{babel}
\usepackage{amsthm}
\swapnumbers
\theoremstyle{plain}

\newtheorem{prop-s}[subsection]{Proposition}
\newtheorem{lemma-s}[subsection]{Lemma}
\newtheorem{obs-s}[subsection]{Observation}
\newtheorem{prop}[subsubsection]{Proposition}
\newtheorem{lemma}[subsubsection]{Lemma}
\newtheorem{fact}[subsubsection]{Fact}

\newtheorem{cor}[subsubsection]{Corollary}
\newtheorem*{claim}{Claim}
\newtheorem*{sia}{Directed Sparse Incomparability Lemma}
\theoremstyle{definition}
\newtheorem*{note}{Note}

\numberwithin{equation}{subsection}

\usepackage{amssymb,amsmath}

\def\setof#1#2{\{#1 \ | \ #2\} }  \def\set#1{\{#1  \} }
 \def\id{\text{id}}
\def\dr{{\downarrow}}
\def\ur{{\uparrow}}    
\def\sue{\subseteq}\def\supe{\supseteq}     \def\wt{\widetilde}
    
\def\ems{\emptyset}         \def\m{\wedge}    \def\bim{\bigwedge}
\def\qtq#1{\quad\text{#1}\quad}

\def\C{\mathcal C}

\def\A{\mathcal A}

\def\B{\mathcal B}

\def\csp{\text{\bf CSP}}
\def\forb{\text{\bf Forb}}

\def\wld{\hbox{$\mathfrak{wld}$}}
\def\cn{\text{\sf Cn}}
\def\cni{_\cn}

\def\ol{\overline}

              \def\smin{\setminus} 
\let\noto\nrightarrow
\DeclareMathSymbol{\heyt}{\mathbin}{symbols}{"29}
\DeclareMathOperator{\Min}{Min}

\let\trleq\preccurlyeq
\let\trgeq\succcurlyeq
\let\trless\prec

\title[Dualities in Heyting algebras]{Dualities and dual pairs\\ in Heyting algebras}

\author{J. Foniok, J. Ne\v set\v ril,  A. Pultr and C. Tardif}

\address{Institute for Operations Research, ETH Zurich, 8092 Zurich, Switzerland}
\email{foniok@math.ethz.ch}

\address{Department of Applied Mathematics and ITI, MFF, Charles University, 
CZ 11800 Praha 1, Malostransk\'e n\'am. 25}
\email{nesetril@kam.mff.cuni.cz, pultr@kam.mff.cuni.cz}

\address{Royal Military College of Canada \\
PO Box 17000 Station ``Forces'' \\
Kingston, Ontario\\
Canada, K7K 7B4 }
\email{Claude.Tardif@rmc.ca}

\thanks{Thanks go  to the project 1M0545  of the Ministry of Education
of the Czech Republic. The fourth author is supported by grants from
the Natural Sciences and Engineering Research Council of Canada and the
Academic Research Program.}

\date{\today}


\begin{document}

\begin{abstract}
We extract the abstract core of finite homomorphism dualities using the techniques of Heyting algebras and (combinatorial) categories.
\end{abstract}

\maketitle

\section*{Introduction}

Finite dualities appeared in~\cite{NesPul:SubFac} in the categorical context
of dual characterizations of various classes of structures. It is a
simple idea: characterize a given class both by forbidden substructures
(associated with subobjects) and by decompositions (associated with
factorobjects); this  proved to be surprisingly fruitful. In retrospect,
it was also a timely concept as it coincided with the introduction
(in the logical and artificial intelligence contexts) of the paradigm
of Constraint Satisfaction (\cite{Mac:Cons,Mon:Netw}).

Only later it was realized (in the context of complexity theory) that these notions are two
aspects of the same general problem, the study of homomorphisms of relational structures
(\cite{FedVar:SNP}).

Finite dualities represent an extremal case of the above mentioned
Constraint Satisfaction Problem. Provided we have a finite duality,
the problem in question is polynomially decidable.  Furthermore,
in a broad context such problems coincide with the decision
problem for classes of structures that are first-order decidable
(\cite{Ats:On-digraph,LarLotTar:A-Characterisation,Ros:preservation}). For
general relational structures, finite dualities were characterized
in~\cite{NesTar:Dual} and a number of interesting particular cases were
investigated as well (\cite{FNT:GenDu,HelNesZhu:DualPoly,Nes:Many,NesTar:Short}).

Here, following \cite{NesPulTar:HeytDual} we return to the original motivation  and discuss finite dualities in the categorical context. We aim at pointing out those categories in which one can describe finite dualities using the interplay of general categorical and order theoretical concepts and techniques.

\vskip10mm

\section{Background: dualities in graphs\\
and similar categories}

\subsection{}\label{1.1}
The categories we will work with are {\em finitely concrete}, that is, the objects are finite sets endowed with structures, and morphisms are maps respecting the structures in a specified way.

Typically we have in mind categories such as that of (finite) symmetric graphs, or oriented
graphs, with edge preserving homomorphisms (or more general relations resp.\ relational
systems), with relation preserving maps. Some of the results can be applied for other choices
of morphisms (strong or full homomorphisms).

\subsection{}\label{1.2}

We will assume that our categories admit finite sums (coproducts)
$$
\iota_j:A_j\to A=\coprod_{i=1}^nA_i
$$
(characterized by the property that for each system $f_j:A_j\to B$, $j=1,\dots,n$ there is exactly one morphism $A\to B$ such that $f\iota_j=f_j$ for all $j$).

This is the minimal assumption; for more involved facts we will assume also the existence of finite products
$$
\pi_j:A=\prod_{i=1}^nA_i\to A_j 
$$
(characterized by the property that for each system $f_j:B\to A_j$ there is exactly
one morphism $B\to A$ such that $\pi_jf=f_j$ for all $j$), and also the Heyting
property (see~\ref{1.6} below).

\subsection{}\label{1.3}

The {\em Constraint Satisfaction Problem} (briefly, \emph{CSP}) in a
category~$\C$ is the membership problem of the class
$$
\csp(\B)=\setof{X\in\C}{X \to B\ \text{for some} \ B\in\B}
$$
where $X\to Y$ stands for ``there exists a morphism $f:X\to Y$'' and $\B$~is a class of objects. Here we are concerned with the situation in which this class can be represented as
$$
\forb(\A)=\setof{X\in\C}{A\noto X\ \text{for all} \ A\in\A}
$$
where $X\noto Y$ stands for ``there exists no morphism $f:X\to Y$'' and $\A$ is a finite class of objects
(with an infinite $\A$ this is always possible). In fact the classes $\B$ we are interested in are also finite. Thus, we investigate the situations of finite systems $A_1,\dots,A_n$ and $B_1,\dots,B_m$ of objects such that
\begin{equation}
\forall i,\ A_i\noto X \qtq{iff} \exists j,\ X\to B_j \label{1.3.1}
\end{equation}
and in this case we speak of a {\em finite duality}.

\begin{note}
Instead of forbidding morphisms from the objects $A_i$, one is sometimes interested
in forbidding  subobjects from a finite family of isomorphism classes. If we
have \eqref{1.3.1}, it is easy to replace the $A_i$'s by finitely many other
objects providing such a ``subobject forbidding characterization''.
Similarly, the $X\to B$ type requirements can typically be replaced by requirements
of epimorphisms.
\end{note}

\begin{note}
The name \emph{Constraint Satisfaction Problem} originates in the
computational setting, where the description involves variables and
constraints (the elements and structures of~$X$) and a domain with
relations (the structures in~$\B$).
\end{note}

\subsection{The poset $\wt\C$}\label{1.4}
Given a category $\C$, consider the set of objects ordered by
$$
A\leq B\quad\equiv_{\text{df}}\quad \exists \ f:A\to B
$$
and denote the obtained (pre-)ordered set by
$$
\wt\C.
$$
In fact, we usually think of $\wt\C$ as the poset of the obvious equivalence classes.

Note that if we assume the existence of sums as we did in~\ref{1.2} above,
$\wt \C$~is a join-semilattice. If we have, moreover, also the products,
$\wt \C$~is a lattice.

\subsection{Heyting categories}

A \emph{Heyting algebra} is a bounded lattice with an extra operation $\heyt$
satisfying
$$
a\m b\leq c\qtq{iff} a\leq (b\heyt c).
$$ 
A \emph{Heyting category} is a category $\C$ such that $\wt\C$ is a Heyting
algebra.

\begin{note}
Trivially, any  Cartesian closed category (that is, a category with exponentiation $\langle X,Y\rangle$ such that the sets of morphisms
$$
A\times B\to C\qtq{and} A\to\langle B,C\rangle
$$
are naturally equivalent~-- see, e.g.~\cite{Mac:Cat}) is Heyting. However, Cartesian
closedness is not in general necessary, because in a Heyting category the requirement
is weaker: we just require that there exist a morphism $A\times B\to C$ iff there
exists a morphism $A\to (B\heyt C)$ and drop the requirement of natural
equivalence. An example of a Heyting category which is not Cartesian closed is the
category of loopless graphs (with an additional terminal object). For this and more examples,
see~\cite{NesPulTar:HeytDual}.
\end{note}

\subsection{Cores}\label{1.6}
In our categories the objects can be canonically reduced to make  the
relation $A\to B$ antisymmetric (up to isomorphism).

\begin{lemma}\label{1.6.1}
Let $X$ be a finite object in a concrete category. Then each bijective
morphism $\phi:X\to X$ is an isomorphism.
\end{lemma}

\begin{proof}
There is a $k\neq 0$ and $n$ such that $\phi^{n+k}=\phi^n$. Thus,
$\phi^k=\id$ and $\phi\cdot\phi^{k-1}= \phi^{k-1}\cdot\phi=\id$.
\end{proof}

\begin{prop}\label{1.6.2}
Let $X$ be a finite object in a concrete category. Then the smallest subobject $Y\sue X$ such that there is a morphism $f:X\to Y$
\begin{enumerate}
\item is a retract of $X$, and
\item is uniquely determined, up to isomorphism.
\end{enumerate}
\end{prop}

\begin{proof}
(1)~Let $j:Y\to X$ be the embedding morphism. Then, by minimality,
$\phi=fj:Y\to Y$ is bijective and by~\ref{1.6.1} it is an isomorphism, and
we have the retraction $r=\phi^{-1}f$. (2)~Now if $j':Z\to X$ is an
embedding of another subobject with the property, we have mutually
inverse isomorphisms $r'j$ and $rj'$.
\end{proof}

The object~$Y$ from~\ref{1.6.2} is called the {\em core} of~$X$; denote
it by~$cX$. An object~$X$ is called a {\em core} if it is the core of
some object.  Note that
\begin{itemize}
\item[--] a core is the core of itself, and
\item[--] in a concrete category $\C$ with finite objects, $A$ and $B$ are equivalent in $\wt\C$ (that is $A\leq B$ and $B\leq A$) iff $cA$ and $cB$ are isomorphic.
\end{itemize}
Thus, if we restrict ourselves to cores and representatives of isomorphism classes,
\begin{itemize}
\item[--] the pre-ordered set $\wt\C$ becomes actually a poset.
\end{itemize}
Furthermore, any duality
$$
\forall i,\ A_i\noto X \qtq{iff} \exists j,\ X\to B_j
$$
can be replaced by the duality in the cores
$$
\forall i,\ cA_i\noto X \qtq{iff} \exists j,\ X\to cB_j
$$
(\cite{HelNes:Core,HelNes:GrH}).


\section{Transversals and  weak right duals}

\subsection{}

In a poset $(X,\leq)$ we will use the standard notation, for a subset $M\sue X$,
\begin{align*}
\dr M &= \setof{x\in X}{\exists m\in M,\ x\leq m},\\
\ur M &= \setof{x\in X}{\exists m\in M,\ x\geq m},
\end{align*}
and for an element $m\in X$, let $\dr m=\dr\set{m}$ and $\ur m=\ur\set{m}$.

\subsection{Connected elements}

A  element $a$ of a lattice $L$ is {\em connected} if 
\begin{equation}
\label{eq:connected}
a\leq b\vee c\quad\Rightarrow\quad a\leq b\ \ \text{or}\ \ a\leq c.
\end{equation}

\begin{note}
Another (and perhaps more frequently used) term is {\em join-prime} or
{\em $\vee$-prime}. We use ``connected'' because of the interpretation
in the posets $\wt \C$ (recall~\ref{1.4}) we are primarily interested
in. Note that a core graph satisfies~\eqref{eq:connected} (in any choice
of morphisms at least as demanding as the standard graph homomorphism)
iff it is connected in the usual sense, i.e., there is a path between any
two of its vertices.
\end{note}
 
The set of all connected elements of a semilattice $L$ will be denoted by
$$
\cn L\qtq{or simply by} \cn\ .
$$

\subsection{Connected decompositions}

The upper semilattices we will be working with will possess \emph{finite
connected decompositions}, that is
\begin{equation}\label{2.3.1}
\text{for each } a\in L
\text{, there is finite } F\subseteq\cn L \text{ such that } a=\textstyle\bigvee F.
\end{equation}

\begin{note}
In a semilattice~$L$ satisfying~\eqref{2.3.1}, the set $\dr a\cap\cn L$ has only
finitely many maximal elements for every $a\in L$, and
\[a=\textstyle\bigvee\max(\dr a\cap\cn).\]
The latter is then the only irredundant connected decomposition of~$a$.
\end{note}

A {\em connected component} of an element $a\in L$ is a $c\in\cn L$
such that $a=c$ or there exists a decomposition $a=b\vee c\neq b$.

For a subset $A\sue L$ write
$$
A\cni
$$
for the set of all connected components of the elements from~$A$.
Clearly, a connected element has exactly one connected component: itself.
Moreover, in a lattice with finite connected decompositions, $A_\cn$~is finite for
every finite subset $A\subseteq L$.

\begin{note}
In a (semi)lattice satisfying~\eqref{2.3.1},
we have \[\{a\}_\cn = \max(\dr a \cap\cn).\]
\end{note}

\subsection{Finite dualities}

A {\em duality pair} $(l,r)$ in $L$ is a pair of elements such that
$$
\dr r= L\smin\ur l
$$
(that is,
$$
l\nleq x\qtq{iff} x\leq r.\ )
$$
The element $r$ (obviously uniquely determined by $l$) is then called the
{\em right dual} of $l$, and similarly  $l$~is called the {\em left dual}
of $r$. If the other element of the pair is not specified we speak of
a {\em right} or a {\em left dual}.

\begin{note}
The elements $l$ that are left duals are always connected. In fact, whenever $l\leq
\bigvee x_i$ for any join $\bigvee x_i$,  then necessarily $l\leq x_j$ for some $j$ (indeed, if for all $j$, $l\nleq x_j$ then for all $j$,
$x_j\leq r$, and hence $\bigvee x_i\leq r$ and $l\nleq\bigvee x_i$). This
property is usually called {\em supercompactness}, for obvious reasons.
\end{note}

A {\em finite duality} in $L$ is a pair $(A,B)$ of finite subsets of $L$ such that
\begin{enumerate}
\item distinct elements in $A$ resp. $B$ are incomparable, and
\item  $x\in \ur A$ iff $x\notin\dr B$ (in other words, $\dr B=L\smin \ur A$).
\end{enumerate}
(Compare with~\ref{1.3}.)
An element $l$ (resp. $r$) is a {\em weak left dual} (resp. {\em
weak right dual}) if there is a finite duality $(A,B)$ such that $l\in A$
(resp. $r\in B$).

\begin{fact}\label{2.4.1}
The set $B$ in a finite duality $(A,B)$ is uniquely determined by $A$,
and vice versa.
\end{fact}

\begin{proof}
If $(A,B_1), (A,B_2)$ are finite dualities then $\dr B_1=\dr B_2$, and hence for each $x\in B_1$ there is an $\alpha(x)\in B_2$ such that $x\leq\alpha(x)$, and similarly for each $x\in B_2$ there is a $\beta(x)\in B_1$ such that $x\leq\beta(x)$. Thus, $x\leq\beta\alpha(x)$ and  $x\leq\alpha\beta(x)$ and by incomparability $\alpha\beta=\id$ and $\beta\alpha=\id$, and finally $x\leq\alpha(x)\leq x$ and $\alpha(x)=x$ and similarly $\beta(x)=v$.
\end{proof}

\subsection{Transversals}

For antichains $M,N\subseteq L$, we will write
\begin{equation}
M\trleq N \qtq{for} N\sue \ur M\ . \label{2.5.1}
\end{equation}
One sometimes speaks of $M$ as of a {\em refinement} of $N$; thus
if $M\trleq N$, then $N$ is {\em coarser}.

A subset $M\sue A\cni$ is  said to be a {\em transversal} of $A\sue L$ if
\begin{itemize}
\item[(T1)] distinct elements of $M$ are incomparable,
\item[(T2)] $A\sue \ur M$, and
\item[(T3)] in the refinement order~$\trleq$, $M$ is maximal with respect to the properties
(T1) and (T2).
\end{itemize}
A subset satisfying only (T1) and (T2) is called a {\em quasitransversal} of~$A$.

\subsection{}

Let $(A,B)$ be a finite duality.
If $M$ is a quasitransversal of~$A$, then by (T2)
$$
\ur A\sue \ur M \qtq{and hence} L\smin \ur M\sue \dr B\ .
$$

For a quasitransversal $M$ of~$A$, set
$$
\ol M=A_\cn\smin\ur M.
$$
Note that if $M,N$ are quasitransversals such that $M\trleq N$, then ${\ol M\subseteq\ol N}$.

In the following, if we speak about $M$~being a \emph{(quasi)trans\-ver\-sal
for} $(A,B)$, we mean that $(A,B)$ is a finite duality and $M$~is a
(quasi)trans\-ver\-sal for~$A$.

\begin{lemma}\label{2.6.1}
Let $\ol M=\ems$ for a transversal $M$ of $(A,B)$. Then
\begin{itemize}
\item[1.] $A=M=A_\cn$, and
\item[2.] $B$ has only one element.
\end{itemize}
\end{lemma}

\begin{proof}
1. Suppose first that $A\ne A\cni$. Let $c$ be a maximal element of the
set $\setof{c\in A\cni}{\text{$c$ is a component of some disconnected
$a\in A$}}$. Then $c\notin A$ because $A$~is an antichain. As $\ol
M=\emptyset$ and by the choice of~$c$, the set $M\setminus\set{c}$
is a quasitransversal, and moreover $M \trleq M\setminus\set{c}$;
thus $M=M\setminus\set{c}$ and so $c\notin M$. Because $c\notin\ol
M=\emptyset$, there is some $c'\in M$ such that $c'<c$. Again, since
$A$~is an antichain, $c'\notin A$. Let $A'=\setof{a\in A}{\dr a\cap
M=\set{c'}}$; observe that $c'\notin A'\cni$ because $c'\notin A$. Let
$M'=\Min\bigl((M\smin\set{c'})\cup A'\cni\bigr)$. Then $M'$~is a
quasitransversal and $M\trless M'$, a contradiction. Thus $A=A\cni$.

Hence $A\cni$ is an antichain, so it is a quasitransversal.  Since
${A_\cn\smin\ur M}=\ems$ we have $A_\cn\sue\ur M$, and hence $M\trleq
A_\cn$.  By~(T3), $A_\cn=M$.

\smallskip

2. Let $r_1,r_2\in B$ be distinct. Then, by the incomparability condition,
$r_1\vee r_2\nleq r$ for all $r\in B$ and hence $l\leq r_1\vee r_2$
for some $l\in A$. Now by~1., $l$ is connected and hence $l\leq r_1$
or $l\leq r_2$, a contradiction.
\end{proof}

\begin{lemma}
Let $M$ be a transversal of $(M,B)$. Then
\begin{itemize}
\item[1.] $\ol M=\emptyset$, and
\item[2.] there is no other transversal.
\end{itemize}
\end{lemma}

\begin{proof}
1. By definition, all elements of a transversal are connected, so
$M=M_\cn$ and $\ol M = M_\cn \setminus \ur M_\cn = \emptyset$.

2. If $N$ is another transversal, then $N\sue M_\cn=M\sue \ur M$, and so
$M\trleq N$. Thus, by (T3), $N=M$.
\end{proof}

\begin{lemma}\label{2.6.3}
Let $M$ be a transversal of $(A,B)$. Then there is precisely one $r\in B$ such that
\begin{enumerate}
\item $M\cap\dr r=\ems$, and
\item $\ol M\sue\dr r$.
\end{enumerate}
\end{lemma}

\begin{proof}
I. If $\ol M=\ems$ for a transversal $M$, we have $B=\set{r}$
by~\ref{2.6.1}, and this $r$ satisfies the conditions.

II. Now suppose that $A$ is not its own transversal. Then $\ol M\neq\ems$.
Set $s=\bigvee\ol M$. We have $s\notin \ur M$ (if $x\in M$ and $x\leq s$ then by
connectedness $x\leq y\in\ol M$), hence $s\notin\ur A$, and consequently $s\in \dr B$
and we have an $r\in B$ such that $s\leq r$ so that $\ol M\sue\dr s\sue \dr r$.

Now suppose $c\in M\cap\dr r$. By (T3), the set $(M\smin\set{c})$ is not
a quasitransversal, and thus $A\nsubseteq\ur(M\smin\set{c})$. Choose
some $l\in A\setminus{\ur(M\smin\set{c})}$ and let $l=\bigvee l_i$ be
a connected decomposition of~$l$. If $l_i\geq b\in M$, then $b=c$ and
hence $l_i=c$.  Hence for every~$i$, either $l_i=c$ or $l_i\in\ol M$,
and therefore $l\leq r$, contradicting the duality.

III. Finally let distinct $r_1,r_2$ have the property. Then $r_1\vee
r_2\notin\dr B$, hence $r_1\vee r_2\in\ur M$ and there is a (connected)
$x\in M$ such that $x\leq r_1\vee r_2$; thus, $x\leq r_1$ or $x\le r_2$,
contradicting~(1).
\end{proof}

\subsection{}\label{2.8}

The uniquely determined $r$ from~\ref{2.6.3} will be denoted by
\[
r(M).
\]
Note that if $A$ is not its own transversal, then $r(M)$~is determined by the formula
\begin{equation}\label{2.7.*}
\textstyle\bigvee\ol M\leq r(M)\in B.
\end{equation}

\begin{lemma}\label{2.8.1}
If $M_1, M_2$ are distinct transversals, then $\ol M_1\cap M_2\neq\ems$.
\end{lemma}

\begin{proof}
If $M_1\neq M_2$, then $M_2\not\trleq M_1$ and hence there is  a $c\in
M_2\smin \ur M_1$. Then $c\in(A_\cn\smin\ur M_1)\cap M_2$.
\end{proof}

\begin{lemma}
For $r\in B$ set $M=\Min\setof{x\in A_\cn}{x\nleq r}$. Then $M$~is a
quasitransversal, and if $M_r$ is a transversal with $M\trleq M_r$,
then $r(M_r)=r$.
\end{lemma}

\begin{proof}
Let $l\in A$. Then $l=\bigvee\setof{x\in A_\cn}{x\leq l}\nleq r$ and
hence there is an $x\in A_\cn$ such that $x\leq  l$ and $x\nleq r$. Thus
$M$~is a quasitransversal.

We have $\ol M_r=\ol M$. If $\ol M=\emptyset$, then \ref{2.6.1}
applies. Suppose that $\ol M\ne \emptyset$ and $\bigvee \ol M\nleq
r$. Then there is an $x\in\ol M$ such that $x\nleq r$, that is, $x\in M$,
a contradiction.
\end{proof}

\begin{prop-s}\label{2.8.3}
Let $(A,B)$ be a finite duality in a semilattice with finite connected
decompositions. The formulas
\begin{align*}
&M\mapsto r(M)\qtq{where} \textstyle\bigvee\ol M\leq r(M)\in B\quad  (\text{if}\ \ol M\neq\ems),\\
&r\mapsto M_r\qtq{where} M_r\trgeq M=\setof{x\in A_\cn}{x\nleq r}
\end{align*}
constitute a one-to-one correspondence between the transversals of $(A,B)$
and elements of $B$.
\end{prop-s}

\begin{proof}
We already know that $r(M_r)=r$. Let $M_1, M_2$ be distinct
transversals. By~\ref{2.8.1}, there is a $c\in\ol M_1\cap
M_2$. By~\ref{2.6.3}(1), $c\nleq \dr r(M_1)$ and by~\eqref{2.7.*},
$c\leq \dr r(M_2)$. Hence $r(M_1)\neq r(M_2)$.
\end{proof}

\begin{prop-s}\label{2.9}
Let $L$ be a semilattice with finite connected decompositions and let
$(A,B)$ be a finite duality.  Then each transversal~$M$ of~$A$ together
with the element $r(M)$ defined in~\ref{2.8} constitutes a finite duality
$\bigl(M,\set{r(M)}\bigr)$.
\end{prop-s}

\begin{proof}
Set $r=r(M)$. We have $M\sue L\smin\dr r(M)$ and hence  $\ur M\sue
L\smin\hbox{$\dr r(M)$}$ by~\ref{2.6.3}.

Now let $x\notin\ur M=\bigcup\setof{\ur c}{c\in M}$. We want to prove
that $x\in\dr r(M)$. Let $y=x\vee\bigvee\ol M$. We have $c\nleq x$ for
all $c\in M$, and by connectedness $c\nleq y$ for all $c\in M$. Suppose
that $l\leq y$ for some $l\in A$. If $c\in M$ and  $c\leq l$ we have
$c\leq y$ and hence $c\leq x$, a contradiction. Thus, $y\notin \ur A$
and hence $y\leq \dr B$, that is, $y\leq r'$ for some $r'\in B$. But
then $\bigvee\ol M\leq r'$ and hence $r'=r(M)$. Therefore $x\in \dr r(M)$.
\end{proof}

\section{Connected components of weak left duals\\ are left duals}

In~\ref{2.8} we have seen that given a finite duality $(A,B)$, each element $r\in B$
is in a duality $(M,\set{r})$. In this section we will obtain dualities in the
reversed order. Instead of dualities for elements $l\in A$ we will have them for
their connected components
$c\in A_\cn$. For these, however, we will prove something stronger. Namely, we will show that each such element is a left dual.

Unlike the previous section we will have to assume that the lattice $L$ is Heyting (and hence the categorical
interpretation holds for Heyting categories only).

\subsection{Gaps}

A pair of elements $(a,b)$ of a poset~$L$ is a \emph{gap} if $a<b$
and $a\le c\le b$ implies that $a=c$ or $b=c$, for every $c\in L$.

\subsection{}
We will need two facts from \cite{NesPulTar:HeytDual}.

\begin{prop}[\cite{NesPulTar:HeytDual} 2.6.]\label{3.1.1}
 The gaps in a Heyting algebra $L$ with connected decompositions  are exactly the pairs $(a,b)$ such that for some
duality $(l,r)$, 
\[
l\m r\leq a\leq r\qtq{and} b=a\vee l.
\]
\end{prop}

\begin{prop}[\cite{NesPulTar:HeytDual} 3.3.]\label{3.1.2}
Let $L$ be a Heyting algebra with connected decompositions, let
$A=\setof{l_i}{i\in J}$ be a subset of~$L$ and let $r\in L$.  Let either
$J$ be finite or $L$ admit infima of sets of the size of~$J$. Then
the pair $(A,\set{r})$ is a duality if and only if there are dualities
$(l_i,r_i)$, $i\in J$, such that
\[
r=\bigwedge_{i\in J}r_i.
\]
\end{prop}

\begin{lemma-s}\label{3.2}
In a Heyting algebra with finite connected decompositions, every element
of a transversal of a finite duality $(A,B)$  is a left dual.
\end{lemma-s}

\begin{proof}
Let $M$ be a transversal. We have the duality $(M,r(M))$,
by~\ref{2.9}. Thus, by~\ref{3.1.2} there is a duality $(m,r_m)$ for each
$m\in M$.
\end{proof}

\begin{prop-s}\label{3.3}
In a Heyting algebra with finite connected decompositions,  a connected
component of a weak left dual is  a left dual.
\end{prop-s}

\begin{proof}
Let $(A,B)$ be a duality and let $c\in A_\cn$. Suppose it is not a left
dual; then in particular, by~\ref{3.2}, it is contained in no transversal.

Set
\[
a=\textstyle\bigvee\setof{c'\in A_\cn}{c'< c}\vee\textstyle\bigvee\setof{c\m c'}{c'\in A_\cn,\ c,c'\ \text{incomparable}}.
\]
We have $ a< c$ since else  by the connectedness of $c$ some of the
summands would be equal to~$c$, which they are not. Now the couple $(a,c)$
is not a gap: else we would have, by~\ref{3.1.1} a duality $(l, r)$
such that $c=a\vee l$, and hence $c=l$.

Thus there exists $x$ such that
\[
a< x< c.
\]

\begin{quotation}
\begin{claim}
If $c'\in A_\cn$ and $c'\neq c$ then
\begin{align*}
c'\leq x &\qtq{iff} c'< c, \qtq{and}\\
c'\geq x &\qtq{iff} c'>c.
\end{align*}
\end{claim}

\begin{proof}[Proof of Claim]
In the first case: if $c'\leq x$, then $c'\leq x< c$; and if $c'<c$, then $c'\leq a<x$.

Now consider the second case. Trivially if $c< c'$, then $x\leq c'$. Now
suppose $x\leq c'$. Then if $c'<c$, we have  $x=c'$ by the first
equivalence, hence  $x=c'\leq a$, a contradiction.  If $c$ and $c'$ are
incomparable, then $x\leq c\m c'\leq a$, a contradiction again. Thus,
$c<c'$ is the only alternative left.
\end{proof}
\end{quotation}

{\em Proof continued.} Let $l\in A$ be such that $c$ is one of its
connected components and let $l=b\vee c\neq b $ be a decomposition
witnessing the fact. Set $q=b\vee x$.  We cannot have $l\leq q$ since
else $c\leq b$ and  $b\leq l\leq b$ contradicting the choice of the
decomposition. Consequently, we also have $l'\nleq q$ for any other $l'\in
A$ since otherwise $l'\leq l$. Thus, $\forall l\in A,\ l\nleq q$ and hence

\[
\exists r\in B, \quad q\leq r.
\]
Let $M$ be the transversal such that $r=r(M)$, so that in particular 
\[
\forall m\in M,\quad m\nleq q.
\]
We have $c\notin M$ since $c$ is in no transversal, and hence $m\nleq x$ for all $m\in M$. By Claim, $m\nleq c$ for all $m\in M$, and hence $c\leq r$. 

Now, since $q\leq r$, we have $l=q\vee c\leq r$ contradicting the duality $(A,B)$.
\end{proof}

\begin{cor}
If a Heyting algebra with finite connected decompositions has no
non-trivial duality pair, then it admits no finite duality.
\end{cor}

\subsection{Note}
Compare the following two facts (the first obtained by combining
\ref{2.8.3} and~\ref{3.1.2}, the second is an immediate consequence
of~\ref{3.3}) holding in Heyting algebras with connected decompositions:
\begin{itemize}
\item[--] each weak right dual is a meet of right duals, and
\item[--] each weak left dual is a join of left duals.
\end{itemize}
(The facts from which these statements follow are, of course, stronger.)

\section{The transversal construction reversed:\\
             from dual pairs to finite dualities}

In previous sections the notion of a transversal helped to analyze
finite dualities $(A,B)$ in terms of the individual elements of $A$
and~$B$. The elements $r\in B$ have been shown to be naturally associated
with transversals of $(A,B)$ (in~\ref{2.8.3}), and then the elements
of~$A$ have been shown to be joins of left duals (see~\ref{3.3}). In this
section we will use the procedure reversely: namely, for a finite set~$A$
of sums of left duals we will construct a finite duality.

\begin{obs-s}\label{4.1}
In any lattice, if $(l_i,r_i)$, $i=1,\dots, n$, are dual pairs, then
$(\set{l_1,\dots,l_n},\set{\bim_{1=1}^nr_i})$ is a duality.
\end{obs-s}

(Indeed, $\forall i, l_i\nleq x$ iff $\forall i, x\leq r_i$ iff $x\leq\bim r_i$.)

\begin{lemma-s}\label{4.3}
In a distributive lattice~$L$, let $c$ be a connected component of
an $a\in L$ and let $a=\bigvee_{i=1}^na_i$. Then $c$ is a connected
component of some~$a_i$.
\end{lemma-s}

\begin{proof}
Let $a=x\vee c\neq x$. By the connectedness, $c\leq a_i$ for
some~$i$. Then $a_i=(x\vee c)\m a_i=(x\m a_i)\vee c\neq x\m a_i$ since
otherwise $c\leq x$ and $a=x$.
\end{proof}

\begin{prop-s}\label{4.4}
Let $L$ be a Heyting algebra with finite connected decompositions. Let
$A$ be a finite set such that each $a\in A$ is a finite join of left
duals. Then there exists a finite duality $(A,B)$.
\end{prop-s}

\begin{proof}

Let $a=\bigvee_{i=1}^{n_a} c_i(a)$ be connected decompositions of the $a\in A$.
 Then
 \[
 A_\cn\sue\setof{c_i(a)}{a\in A,\ i=1,\dots n_a}.
 \]
Now if $a=\bigvee_{j=1}^ka_j$ with $a_j$ left duals, then  each connected
component of~$a$ is, by~\ref{4.3}, a connected component of some of
the~$a_j$, and hence each $c_i(a)\in A_\cn$ is, by~\ref{3.3}, a left
dual. Denote by $r_i(a)$ the corresponding right dual.

 Let $\mathcal M$ be the set of transversals of $A$, hence
 $M\sue A_\cn$. Trivially, it is finite. For $M\in\mathcal M$
 set $$ r_M=\textstyle\bim\setof{r_i(q)}{c_i(q)\in M} $$ and consider $$
 B=\setof{r_M}{M\in\mathcal M}.  $$ Let $x\leq r_M$ for some $M\in\mathcal
 M$. Then for all $c_i(q)\in M$, we have $x\leq r_i(q)$ and hence $c_i(q)\nleq
 x$. For an arbitrary $a\in A$ there is a $c_i(q)\leq a$ and hence
 $a\nleq x$.

 On the other hand, let $a\nleq x$ for all $a\in A$. Thus, for each $a\in
 A$ we have a connected component $x_{i_a}(a)$ such that $x_{i_a}(a)
 \nleq x$. Set $$ M'=\setof{x_{i_a}(a) }{a\in A} $$ and consider $M''$
 the system of all minimal elements of $M'$ (to satisfy~(T1)). Now
 $M''$ is a quasitransversal and we have a transversal $M\trleq M''$. Then
 for each $c_i(q)\in M$, we have $c_i(q)\nleq x$, hence $x\leq r_i(q)$,
 and finally $x\leq r$.
\end{proof}

\begin{note}
The elements $r_M$ corresponding to the transversals~$M$ are exactly
the minimal elements of $B'=\setof{\bim_{a\in A} r_{i_a}(a)}{1\le i_a\le
n_a}$, and so $B=\Min B'$.
\end{note}

From~\ref{3.3} and~\ref{4.4} we now immediately obtain

\begin{cor}
Let $L$ be a Heyting algebra with finite
connected decompositions. Then the following statements on an element
$a\in L$ are equivalent: \begin{enumerate} \item $a$ is a weak left dual,
\item $a$ is a finite join of left duals.
\end{enumerate}
\end{cor}

\subsection{}
Define $$ \wld (L) $$ as the set of all the weak left duals
in~$L$ (this is the obvious abbreviation, but, by coincidence it also
alludes to the German word ``Wald'' for forest; it so happens that in
case of binary relations the weak left duals are precisely the disjoint
unions of trees, the forests). Then, by~\ref{4.4} (and~\ref{2.4.1}) we have

\begin{cor}
For each subset $A\sue\wld(L)$ there is precisely one finite duality $(A,B)$.
\end{cor}

\subsection{Note}

By~\ref{4.1} and the definition of $r_M$, we have the dualities
$(M,\set{r_M})$, and hence if $c_i(q)\in\ol M$, that is, $c_i(q)\notin\ur
M$, then $c_i(q)\leq r_M$. Thus, $$ \textstyle\bigvee\ol M\leq r_M $$
and $r_M=r(M)$ as in~\ref{2.8.3}.

\subsection{Remarks}

As we already mentioned in~\ref{1.1}, typical examples of a
Heyting algebra we have in mind are provided by categories of graphs or
relational structures.
A characterization of finite dualities in the category of relational
structures is provided in~\cite{NesTar:Dual,FNT:GenDu}. The relationship
between finite dualities and duality pairs has recently
been reproved in the special case of digraphs~\cite{ErdSou:No-FiniteInnite}
using the Directed Sparse Incomparability Lemma. Here we would like to
point out that sparse incomparability is \emph{not} necessary to achieve
these results; in fact, much weaker assumptions suffice. However, we do
employ sparse incomparability in the next section, where we study the
connection between dualities and maximal antichains.

\section{Sparse incomparability and antichains}

\subsection{}
In~\cite{DufErdNes:Antichains-in-the-homomorphism} one can
find the following fact (cf.~\cite{NesRod:LGH,NesZhu:On-sparse}).

\begin{sia}
Let $m$, $k$ be positive integers and let $H$ be a directed graph which
is not an orientation of a forest. Then there exists a directed graph~$H'$
such that
\begin{enumerate}
\item the girth of $H'$ is finite and greater than~$k$,
\item for each directed graph $G$ with fewer than $m$~vertices, we have $H'\to G$ iff $H\to G$, and
\item $H\noto H'$ and $H'\to H$.
\end{enumerate}
\end{sia}

It should be  clear now  why the following assumption   will be made in
the Heyting context.

\smallskip
 
 {\bf Sparse incomparability axiom} -- briefly, {\bf SIA}.
 
  This is the assumption that for any $x\in L$, any $M,U$ finite subsets of $L$ such that $(\set{x}\cup\ur U)\cap\wld(L)=\ems$, 
 there is a $y\in L\smin\wld(L)$ such that
 \begin{equation}
 y\in \ur\set{x}, \ \ y\notin \ur(\set{x}\cup U)\qtq{and}\forall m\in M,\ y\leq m\ \text{iff}\  x\leq m. \tag{SIA}
 \end{equation}

So for instance, if $\C$ is the category of digraphs (or relational
structures with a fixed signature) and homomorphisms, then the
poset~$\wt\C$ defined in~\ref{1.4} is a Heyting algebra with SIA.

\begin{obs-s} \label{5.2} If $(A,B)$ is a finite duality in a lattice $L$,
then \[ A \cup (B \smin \dr A) \] is a finite maximal antichain in $L$.
\end{obs-s}

(Indeed, it is an antichain because $a\nleq b$ for any $a\in A$, $b\in B$. 
It is maximal because each  $x\in L$ is either in $\ur A$ or there is a $b\in B$ with $x\leq B$; in the latter case,
if $b\leq a$ for an $a\in A$ we have $x\leq a$.)

\subsection{}

We are now going to show that the antichains of~\ref{5.2} are in some
sense the typical antichains in Heyting algebras with finite connected
decompositions and SIA.

\begin{lemma}\label{5.3.1}
In a Heyting algebra $L$ with SIA and finite connected decompositions, let $C$ be a finite maximal antichain. Set $A=C\cap\dr\wld(L)$. Then
\[
\ur C\smin C=\ur A\smin C.
\]
\end{lemma}

\begin{proof}
The inclusion $\supe$ is trivial. 
Thus, let $x\in\ur C\smin C$ and set $U=C\smin A$. 

If $x\notin\ur U$ then $x\in\ur A$ and hence $x\in\ur A\smin C$. 

If $x\in\ur U$ then $x\notin \wld(L)$. We have  $\ur U\cap\wld(L)=\ems$ and hence we can apply SIA to obtain a $y\notin\wld(L)$ such that
\begin{equation}
y\notin \ur(\set{x}\cup U)\qtq{and}\forall m\in C\cup \set{x},\ y\leq m\ \text{iff}\ x\leq m.  \tag{$*$}
\end{equation}
Now $x\in\ur C\smin C$ and hence if $m\in C\cup\set{x}$ then $x\leq m$ only if $x=m$ and consequently $y\notin\dr C$. Since $C$ is a maximal antichain, $y\in\ur C\smin C$. By $(*)$, $y\notin U=C\smin A$, hence $y\in\ur A$ and since $y\leq x$ we have, by $(*)$ again, $x\in\ur A$.
\end{proof}

\begin{prop}\label{5.3.2}
In a Heyting algebra $L$ with SIA and finite
connected decompositions, let $C$ be a finite maximal antichain such that
$A=C\cap\dr\wld(L)=C\cap\wld(L)$. Consider the unique finite duality
$(A,B)$. Then
\[
C=A\cup(B\smin\dr A).
\]
\end{prop}

\begin{proof}
If $b\in B$ then $b\notin\ur A$ and hence, by~\ref{5.3.1}, $b\notin \ur
C\smin C$. Consequently, since $C$ is a maximal antichain, $b\in \dr C$.

Now suppose that, moreover, $b\notin\dr A$. We want to prove that $b\in
C$. If not, $b<c$ for some $c\in C$ and this means, by our assumption,
that $c\in C\smin A$. Then $c\nleq b'$ for all $b'\in B$ and hence,
by duality, $a\leq c$ for some $a\in A $ contradicting the antichain
property. Thus, $b\in C$ and we have $A\cup(B\smin\dr A)\sue C$,
and since by~\ref{5.2} $A\cup(B\smin\dr A)$ is a maximal antichain,
$A\cup(B\smin\dr A)= C$.
\end{proof}

\subsection{Splitting antichains}

Our final Proposition~\ref{5.3.2} asserts
that every finite maximal antichain~$C$ in a Heyting algebra~$L$ with SIA
either contains an element of $\dr\wld(L)\setminus\wld(L)$, or has a very
special structure. In particular, if $B\cap\dr A=\emptyset$, it can
be \emph{split} into two subsets ($A$ and $B$) so that every element
of~$L\setminus C$ is above~$A$ or below~$B$. This fact has a direct
connection to the splitting property of posets
(see~\cite{AhlErdGra:A-splitting,AhlKha:Splitting,Dza:A-Note,Erd:Splitting,ErdSou:How-to-split,ErdSou:No-FiniteInnite,FonNes:Splitting-finite}).
Indeed, in~\cite{FNT:GenDu} it is proved that finite maximal
antichains in the poset arising from the category~$\C$ of
relational structures with one relation contain no elements of
$\dr\wld(\widetilde\C)\setminus\wld(\widetilde\C)$. Almost all such
antichains split, with finitely many characterized exceptions.  However,
the proof relies heavily on special properties of the category in question
(digraphs) and we do not expect that it could be easily translated into
the general setting of Heyting algebras.

\subsection{Remarks}

The Sparse Incomparability Lemma has a long
and interesting history. While it seems to have been formulated
specifically in this form first in~\cite{NesRod:LGH} for $G=K_n$ and
then in~\cite{NesZhu:On-sparse} for general $G$, it was preceded in
the seminal work on sparse graphs with high chromatic number  by Erd\H
os and others (\cite{BolSau:Uniquely-colourable, Erd:Gtp, ErdHaj:LGH,
GreLov:Applications, Kri:A-hypergraph-free, Lov:LGH, MatNes:Construction,
Mul:On-colorable, Mul:On-colorings, Zhu:Uniquely-H-colorable}). This
useful lemma is related to an important result in descriptive complexity
by Kun (\cite{Kun:Constraints}), and to a recent result on limits
in graph sequences (\cite{NesOss:ResDual,NesOss:From}).



\begin{thebibliography}{10}

\bibitem{AhlErdGra:A-splitting}
R.~Ahlswede, P.~L. Erd{\H o}s, and N.~Graham.
\newblock A splitting property of maximal antichains.
\newblock {\em Combinatorica}, 15(4):475--480, 1995.

\bibitem{AhlKha:Splitting}
R.~Ahlswede and L.~H. Khachatrian.
\newblock Splitting properties in partially ordered sets and set systems.
\newblock In I.~Alth{\"o}fer, N.~Cai, G.~Dueck, L.~H. Khachatrian, M.~S.
  Pinsker, A.~S{\'a}rk{\"o}zy, I.~Weneger, and Z.~Zhang, editors, {\em Numbers,
  Information and Complexity}, volume~12, pages 29--44. Kluwer Academic
  Publishers, Boston, 2000.

\bibitem{Ats:On-digraph}
A.~Atserias.
\newblock On digraph coloring problems and treewidth duality.
\newblock {\em European J. Combin.}, 29(4):796--820, 2008.

\bibitem{BolSau:Uniquely-colourable}
B.~Bollob{\'a}s and N.~Sauer.
\newblock Uniquely colourable graphs with large girth.
\newblock {\em Canad. J. Math.}, 28(6):1340--1344, 1976.

\bibitem{DufErdNes:Antichains-in-the-homomorphism}
D.~Duffus, P.~L. Erd{\H o}s, J.~Ne{\v s}et{\v r}il, and L.~Soukup.
\newblock Antichains in the homomorphism order of graphs.
\newblock {\em Comment. Math. Univ. Carolin.}, 48(4):571--583, 2007.

\bibitem{Dza:A-Note}
M.~D{\v z}amonja.
\newblock A note on the splitting property in strongly dense posets of size
  {$\aleph_0$}.
\newblock {\em Rad. Mat.}, 8(2):321--326, 1992/98.

\bibitem{Erd:Gtp}
P.~Erd{\H o}s.
\newblock Graph theory and probability.
\newblock {\em Canad. J. Math.}, 11:34--38, 1959.

\bibitem{ErdHaj:LGH}
P.~Erd{\H o}s and A.~Hajnal.
\newblock On chromatic number of graphs and set-systems.
\newblock {\em Acta Math. Acad. Sci. Hungar.}, 17:61--99, 1966.

\bibitem{Erd:Splitting}
P.~L. Erd{\H o}s.
\newblock Splitting property in infinite posets.
\newblock {\em Discrete Math.}, 163(1):251--256, 1997.

\bibitem{ErdSou:How-to-split}
P.~L. Erd{\H o}s and L.~Soukup.
\newblock How to split antichains in infinite posets.
\newblock {\em Combinatorica}, 27(2):147--161, 2007.

\bibitem{ErdSou:No-FiniteInnite}
P.~L. Erd{\H o}s and L.~Soukup.
\newblock No finite--infinite antichain duality in the homomorphism poset of
  directed graphs.
\newblock {\em Order}, to appear, 2009.

\bibitem{FedVar:SNP}
T.~Feder and M.~Y. Vardi.
\newblock The computational structure of monotone monadic {SNP} and constraint
  satisfaction: {A} study through {D}atalog and group theory.
\newblock {\em SIAM J. Comput.}, 28(1):57--104, 1999.

\bibitem{FonNes:Splitting-finite}
J.~Foniok and J.~Ne{\v s}et{\v r}il.
\newblock Splitting finite antichains in the homomorphism order.
\newblock In Y.~Boudabbous and N.~Zaguia, editors, {\em Proceedings of the
  first International Conference on Relations, Orders and Graphs: Interaction
  with Computer Science (ROGICS'08)}, pages 327--332, 2008.

\bibitem{FNT:GenDu}
J.~Foniok, J.~Ne{\v s}et{\v r}il, and C.~Tardif.
\newblock Generalised dualities and maximal finite antichains in the
  homomorphism order of relational structures.
\newblock {\em European J. Combin.}, 29(4):881--899, 2008.

\bibitem{GreLov:Applications}
D.~Greenwell and L.~Lov{\'a}sz.
\newblock Applications of product colouring.
\newblock {\em Acta Math. Acad. Sci. Hungar.}, 25:335--340, 1974.

\bibitem{HelNes:Core}
P.~Hell and J.~Ne{\v s}et{\v r}il.
\newblock The core of a graph.
\newblock {\em Discrete Math.}, 109:117--126, 1992.
\newblock Dedicated to Gert Sabidussi on the occasion of his 60th birthday.

\bibitem{HelNes:GrH}
P.~Hell and J.~Ne{\v s}et{\v r}il.
\newblock {\em Graphs and Homomorphisms}, volume~28 of {\em Oxford Lecture
  Series in Mathematics and Its Applications}.
\newblock Oxford University Press, 2004.

\bibitem{HelNesZhu:DualPoly}
P.~Hell, J.~Ne{\v s}et{\v r}il, and X.~Zhu.
\newblock Duality and polynomial testing of tree homomorphisms.
\newblock {\em Trans. Amer. Math. Soc.}, 348(4):1281--1297, 1996.

\bibitem{Kri:A-hypergraph-free}
I.~K{\v r}{\'\i}{\v z}.
\newblock A hypergraph-free construction of highly chromatic graphs without
  short cycles.
\newblock {\em Combinatorica}, 9(2):227--229, 1989.

\bibitem{Kun:Constraints}
G.~Kun.
\newblock Constraints, {MMSNP} and expander relational structures.
\newblock 15 pages manuscript, arXiv: 0706.1701.

\bibitem{LarLotTar:A-Characterisation}
B.~Larose, C.~Loten, and C.~Tardif.
\newblock A characterisation of first-order constraint satisfaction problems.
\newblock {\em Log. Methods Comput. Sci.}, 3(4:6):1--22, 2007.

\bibitem{Lov:LGH}
L.~Lov{\'a}sz.
\newblock On chromatic number of finite set-systems.
\newblock {\em Acta Math. Hungar.}, 19({1--2}):59--67, 1968.

\bibitem{Mac:Cat}
S.~Mac~Lane.
\newblock {\em Categories for the Working Mathematician}, volume~5 of {\em
  Graduate texts in mathematics}.
\newblock Springer-Verlag, New York Berlin Heidelberg, 2nd edition, 1998.

\bibitem{Mac:Cons}
A.~K. Mackworth.
\newblock Consistency in networks of relations.
\newblock {\em Artificial Intelligence}, 8(1):99--118, 1977.

\bibitem{MatNes:Construction}
J.~Matou{\v s}ek and J.~Ne{\v s}et{\v r}il.
\newblock Construction of sparse graphs with given homomorphisms.
\newblock Manuscript.

\bibitem{Mon:Netw}
U.~Montanari.
\newblock Networks of constraints: {F}undamental properties and applications to
  picture processing.
\newblock {\em Inform. Sci.}, 7:95--132, 1974.

\bibitem{Mul:On-colorable}
V.~M{\"u}ller.
\newblock On colorable critical and uniquely colorable critical graphs.
\newblock In {\em Recent advances in graph theory (Proc. Second Czechoslovak
  Sympos. on Graph Theory)}, pages 385--386, Prague, 1975. Academia.

\bibitem{Mul:On-colorings}
V.~M{\"u}ller.
\newblock On colorings of graphs without short cycles.
\newblock {\em Discrete Math.}, 26(2):165--176, 1979.

\bibitem{Nes:Many}
J.~Ne{\v s}et{\v r}il.
\newblock Many facets of dualities.
\newblock In W.~J. Cook, L.~Lov{\'a}sz, and J.~Vygen, editors, {\em Research
  Trends in Combinatorial Optimization}, pages 285--302. Springer, Berlin
  Heidelberg, 2009.

\bibitem{NesOss:ResDual}
J.~Ne{\v s}et{\v r}il and P.~Ossona~de Mendez.
\newblock Grad and classes with bounded expansion. {III}. {R}estricted graph
  homomorphism dualities.
\newblock {\em European J. Combin.}, 29(4):1012--1024, 2008.

\bibitem{NesOss:From}
J.~Ne{\v s}et{\v r}il and P.~Ossona~de Mendez.
\newblock From sparse graphs to nowhere dense structures: Decompositions,
  independence, dualities and limits.
\newblock In {\em Proc. 5th European C. Mathematicians}, 2009.
\newblock To appear.

\bibitem{NesPul:SubFac}
J.~Ne{\v s}et{\v r}il and A.~Pultr.
\newblock On classes of relations and graphs determined by subobjects and
  factorobjects.
\newblock {\em Discrete Math.}, 22(3):287--300, 1978.

\bibitem{NesPulTar:HeytDual}
J.~Ne{\v s}et{\v r}il, A.~Pultr, and C.~Tardif.
\newblock Gaps and dualities in {H}eyting categories.
\newblock {\em Comment. Math. Univ. Carolin.}, 48(1):9--23, 2007.

\bibitem{NesRod:LGH}
J.~Ne{\v s}et{\v r}il and V.~R{\"o}dl.
\newblock A short proof of the existence of highly chromatic hypergraphs
  without short cycles.
\newblock {\em J. Combin. Theory Ser. B}, 27(2):225--227, 1979.

\bibitem{NesTar:Dual}
J.~Ne{\v s}et{\v r}il and C.~Tardif.
\newblock Duality theorems for finite structures (characterising gaps and good
  characterisations).
\newblock {\em J. Combin. Theory Ser. B}, 80(1):80--97, 2000.

\bibitem{NesTar:Short}
J.~Ne{\v s}et{\v r}il and C.~Tardif.
\newblock Short answers to exponentially long questions: Extremal aspects of
  homomorphism duality.
\newblock {\em SIAM J. Discrete Math.}, 19(4):914--920, 2005.

\bibitem{NesZhu:On-sparse}
J.~Ne{\v s}et{\v r}il and X.~Zhu.
\newblock On sparse graphs with given colorings and homomorphisms.
\newblock {\em J. Combin. Theory Ser. B}, 90(1):161--172, 2004.

\bibitem{Ros:preservation}
B.~Rossman.
\newblock Existential positive types and preservation under homomorphisms.
\newblock In {\em Proceedings of the 20th IEEE Symposium on Logic in Computer
  Science (LICS'05)}, pages 467--476. IEEE Computer Society, 2005.

\bibitem{Zhu:Uniquely-H-colorable}
X.~Zhu.
\newblock Uniquely {$H$}-colorable graphs with large girth.
\newblock {\em J. Graph Theory}, 23(1):33--41, 1996.

\end{thebibliography}
\end{document}